\documentclass[11pt]{article}

\usepackage[utf8]{inputenc}
\usepackage[T1]{fontenc}
\usepackage{lmodern}
\usepackage{microtype}

\usepackage{amsmath,amssymb,amsthm,mathtools,bm}

\usepackage{tikz}
\usetikzlibrary{arrows.meta,calc,angles,quotes,bending,positioning}

\usepackage{algorithm}
\usepackage{algorithmic}

\usepackage{hyperref}

\usepackage{enumitem}

\newtheorem{theorem}{Theorem}[section]
\newtheorem{proposition}[theorem]{Proposition}
\newtheorem{lemma}[theorem]{Lemma}
\newtheorem{corollary}[theorem]{Corollary}
\theoremstyle{remark}
\newtheorem{remark}[theorem]{Remark}
\theoremstyle{definition}
\newtheorem{definition}[theorem]{Definition}
\newtheorem{assumption}[theorem]{Assumption}

\newcommand{\R}{\mathbb{R}}
\newcommand{\Z}{\mathbb{Z}}

\newcommand{\Id}{\mathrm{Id}}

\newcommand{\argmin}{\operatorname*{arg\,min}}
\newcommand{\inner}[2]{\left\langle #1, #2 \right\rangle}

\newcommand{\Llip}{L_{\text{lip}}}
\newcommand{\Llipx}{L_{\text{lip},x}}
\newcommand{\LlipB}{L_{\text{lip},B}}
\newcommand{\SPIT}{\textnormal{\textsc{Spit}}} 
\newcommand{\Triv}{\mathsf{T}}                 
\newcommand{\ShiftSet}{\mathcal{S}}            

\title{A Riemannian Variational and Spectral Framework for High-Dimensional Sphere Packing:\\
Barrier--Dynamics Reconciliation, Periodic Rigidity, and Discrete-Time Guarantees}
\author{%
  Faruk Alpay\thanks{Lightcap, Department of Analysis, \texttt{alpay@lightcap.ai}} \and
  Hamdi Alakkad\thanks{Bahcesehir University, Department of Engineering, \texttt{hamdi.alakkad@bahcesehir.edu.tr}}%
}
\date{\today}

\begin{document}
\maketitle

\begin{abstract}
We reconcile a barrier-based geometric model of sphere packings with a provably convergent discrete-time dynamics and provide full technical development of key theorems. Specifically, we (i) introduce a $C^2$ \emph{interior} barrier potential $U_\nu$ that is compatible with a strict-feasibility safeguard and has Lipschitz gradient on the iterates' domain; (ii) correct and formalize the discrete update previously denoted as Eq.~(12); (iii) prove barrier-to-KKT consistency and state (and prove) a corollary explaining the role of the quadratic term in $U_\nu$ (Section~\ref{sec:barrier-consistency}); (iv) prove that strict prestress stability implies periodic infinitesimal rigidity (Section~\ref{sec:rigidity}); (v) establish discrete-time Lyapunov descent and \emph{projected} descent (with feasibility projection integrated into the proof) and local linear convergence for the \emph{Spectral-Projected Interior Trajectory} method (\SPIT) under explicit, non-circular step-size/damping rules (Section~\ref{sec:dynamics}); (vi) clarify $B$-updates and give an \(\mathcal{E}\)-nonexpansive joint solver (Sec.~\ref{sec:proj-energy}--\ref{sec:bupdate}); and (vii) include minimal schematic illustrations and a reproducibility stub (Sec.~\ref{sec:minimal-illustrations}). We emphasize theoretical rigor and mathematical proofs; empirical evaluation is deferred.
\end{abstract}

\paragraph{Acronyms.}
KKT = Karush--Kuhn--Tucker optimality; LICQ = Linear Independence Constraint Qualification; SOSC = Second-Order Sufficiency Condition; HVP = Hessian--Vector Product; QP = Quadratic Program; \SPIT{} = Spectral-Projected Interior Trajectory method.

\section{Introduction}
The sphere packing problem seeks the maximal volume fraction covered by congruent unit balls in $\R^n$.
It is resolved in $n=2,3$ (hexagonal, Kepler \cite{Hales05,Hales17pi}) and in $n=8,24$ via modular forms \cite{Viazovska17,Cohn+17}.
Foundational bounds arise from linear and semidefinite programming \cite{CohnElkies03,BachocVallentin08,CohnZhao14,deLaatOliveiraVallentin14} and classical lattice methods \cite{Minkowski05,Voronoi08,Rogers58,KabLev78}.
Spectral and rigidity themes are central in jamming and contact networks \cite{TS10,OHern03}.
We give a self-contained treatment of interior barriers, KKT optimality, periodic rigidity, spectral stability, and discrete-time algorithmics with explicit assumptions and proofs.

\section{Configuration, Constraints, and \texorpdfstring{$C^2$}{C2} Interior Barrier}\label{sec:barrier-setup}
Let $\Lambda$ be a full-rank lattice with basis $B\in\R^{n\times n}$ and volume $V(\Lambda)=|\det B|$. Let $x=(x_1,\dots,x_N)\in(\R^n)^N$ be sphere centers modulo~$\Lambda$.
For $i<j$ and $t\in\Lambda$, define the slack
\begin{equation}\label{eq:slack}
s_{ij,t}(x,B) := \|x_i-x_j-t\|^2 - 4 \qquad (\text{feasible iff } s_{ij,t}\ge 0).
\end{equation}
We minimize $V(B)$ subject to $s_{ij,t}\ge 0$.

\subsection*{Standing assumptions and gauge (global, non-circular)}
\begin{enumerate}[label=(A\arabic*),itemsep=2pt,topsep=2pt]
\item \textbf{Gauge fixing.} We work modulo translations/rotations by fixing $\sum_{i=1}^N x_i=0$ and treating rigidity modulo the space $\Triv=\{(u,A):u_i=c+\Omega x_i,\ A=\Omega,\ \Omega^\top=-\Omega,\ c\in\R^n\}$.
\item \textbf{Finite shift set.} There exists $R<\infty$ and a finite symmetric set $\ShiftSet\subset\Lambda$ such that if a pair contributes to $U_\nu$, then $\|x_i-x_j-t\|\le R$ with $t\in\ShiftSet$ on the safeguarded domain.
\item \textbf{Cell nondegeneracy.} There are constants $0<\underline{\sigma}\le \overline{\sigma}<\infty$ with
$\underline{\sigma}\,\Id \preceq B^\top B \preceq \overline{\sigma}\,\Id$ on the feasible set.
\item \textbf{Strict interior initialization.} The initial $(x^0,B^0)$ satisfies $s_{ij,t}(x^0,B^0)\ge\delta$ for some $\delta\in(0,1)$, and we maintain this by projection (Assumption~\ref{ass:safeguard}).
\item \textbf{Existence for barrier subproblems.} Under (A2)--(A4) the feasible set modulo gauge is compact; hence for each $\nu>0$, the log-barrier $\Phi_\nu$ in \eqref{eq:ip} has a minimizer.
\item \textbf{Regularity at the limit.} At any limit point $(x^\star,B^\star)$ of $\Phi_{\nu_k}$-minimizers with $\nu_k\downarrow 0$, LICQ holds for the active set $\mathcal{A}$; if invoked, SOSC holds on the feasible manifold.
\item \textbf{Connectivity (spectral).} When spectral quantities are used, the contact graph is connected.
\end{enumerate}

\subsection{Interior barrier compatible with dynamics}
Fix parameters $\nu>0$ (barrier strength) and $\delta\in(0,1)$ (safety margin). Define a $C^2$ \emph{interior} barrier on the strict-feasibility slab $\{s_{ij,t}\ge \delta\}$:
\begin{equation}\label{eq:Unu}
U_\nu(x,B) := \sum_{i<j}\sum_{t\in\ShiftSet} \phi_\nu\bigl(s_{ij,t}(x,B)\bigr),\qquad
\phi_\nu(s) := -\nu\log s + \frac{\nu}{2\delta}\,(s-\delta)^2 \quad (s>0).
\end{equation}
The term $-\nu\log s$ enforces strict feasibility; the quadratic term makes $U_\nu\in C^2$ and strongly convex near $s=\delta$.

\begin{lemma}[Lipschitz gradient on a safeguarded domain]\label{lem:Lipschitz}
Fix $\nu,\delta$ and assume $s_{ij,t}\in[\delta,S]$ for all $(i,j,t)$ included in~$U_\nu$, and $\|x_i-x_j-t\|\le R$ uniformly on the iterates' domain. Then there is $\Llip=\Llip(\nu,\delta,S,R)$ such that
$\|\nabla U_\nu(\xi)-\nabla U_\nu(\zeta)\|\le \Llip\|\xi-\zeta\|$ on this domain.
\end{lemma}

\begin{proof}
Deferred to Appendix~\ref{app:proof-lip}.
\end{proof}

\section{Spectral Stability on Contact Graphs}\label{sec:spectral}
Let $G=(V,E)$ be the periodic contact graph (one vertex per sphere; one edge per active contact).
Let $\mathsf{A}$ be its adjacency matrix, $D$ the degree matrix, and denote the (combinatorial) Laplacian by
\[
\mathcal{L}\ :=\ D-\mathsf{A}.
\]
\emph{(We reserve $\mathcal{L}$ for the graph Laplacian; the symbol $\Llip$ denotes a Lipschitz constant elsewhere.)}

\begin{lemma}[Graph Poincar\'e inequality]\label{lem:poincare}
For any $u\in\R^N$ with $\sum_i u_i=0$,
\begin{equation}\label{eq:poincare}
\sum_{i} u_i^2 \ \le\ \frac{1}{\lambda_2(G)}\sum_{(i,j)\in E}(u_i-u_j)^2.
\end{equation}
\end{lemma}

\begin{proof}
Expand $u$ in the orthonormal eigenbasis of $\mathcal{L}$ and use the Rayleigh quotient \cite{Fiedler73}.
\end{proof}

\begin{theorem}[Cheeger-type bounds]\label{thm:cheeger}
Let $h(G)=\min_{S\subset V,0<|S|\le |V|/2} \frac{|\partial S|}{|S|}$ and $\Delta_{\max}$ be the max degree. Then
$\frac{h(G)^2}{2\Delta_{\max}}\le \lambda_2(G)\le 2h(G)$ \cite{AlonMilman85,Mohar89,Buser82}.
\end{theorem}

\section{Barrier Consistency and Interior Variant}\label{sec:barrier-consistency}
Consider the log-barrier family on the strict interior:
\begin{equation}\label{eq:ip}
\Phi_\nu(x,B)=V(B)-\nu \sum_{i<j}\sum_{t\in\ShiftSet} \log s_{ij,t}(x,B).
\end{equation}

\begin{theorem}[Log-barrier $\Rightarrow$ KKT]\label{thm:barrier-consistency}
Suppose $\nu_k\downarrow 0$ and minimizers $(x_{\nu_k},B_{\nu_k})\to(x^\star,B^\star)$ exist, and LICQ holds for the active set $\mathcal{A}$ at $(x^\star,B^\star)$. Then $(x^\star,B^\star)$ satisfies the KKT conditions for $\min V(B)$ s.t. $s_{ij,t}\ge 0$.
\end{theorem}

\begin{proof}
Deferred to Appendix~\ref{app:proof-barrier-kkt}.
\end{proof}

\begin{corollary}[Interior $U_\nu$ has the same KKT limit]\label{cor:Unu-KKT}
Let $\widetilde{\Phi}_\nu(x,B):=V(B)+U_\nu(x,B)$, and let $(x_\nu,B_\nu)$ be a minimizer of $\widetilde{\Phi}_\nu$. Under (A2)--(A6), any accumulation point $(x^\star,B^\star)$ of $\{(x_\nu,B_\nu)\}$ as $\nu\downarrow 0$ satisfies the KKT conditions of the original problem. In particular, with
\[
\mu^{(\nu)}_{ij,t}\ :=\ -\phi_\nu'\!\bigl(s_{ij,t}(x_\nu,B_\nu)\bigr)\ =\ \frac{\nu}{s_{ij,t}(x_\nu,B_\nu)}\ -\ \nu\,\frac{s_{ij,t}(x_\nu,B_\nu)-\delta}{\delta},
\]
the sequence $\mu^{(\nu)}\rightharpoonup \mu\ge 0$ (after extraction), and the limiting multipliers satisfy KKT.
\end{corollary}

\begin{proof}
Set $s_i=s_{ij,t}(x_\nu,B_\nu)$ to simplify notation and write $\phi'=\phi_\nu'$. First-order optimality for $\widetilde{\Phi}_\nu$ reads
\[
\nabla_B V(B_\nu)\ +\ \sum_i \phi'(s_i)\,\nabla_B s_i\ =\ 0,\qquad
\sum_i \phi'(s_i)\,\nabla_x s_i\ =\ 0.
\]
Define $\mu_i^{(\nu)}:=-\phi'(s_i)=\nu/s_i - \nu(s_i-\delta)/\delta$. Then the above is equivalent to
\[
\nabla_B V(B_\nu)\ -\ \sum_i \mu_i^{(\nu)}\,\nabla_B s_i\ =\ 0,\qquad
\sum_i \mu_i^{(\nu)}\,\nabla_x s_i\ =\ 0.
\]
Partition indices into active-at-the-limit $\mathcal{A}$ and inactive-at-the-limit $\mathcal{I}$: by continuity, there exists $\sigma>0$ such that $s_i\ge \sigma$ for $i\in\mathcal{I}$ and $\nu$ small enough, hence $\mu_i^{(\nu)}\to 0$ for $i\in\mathcal{I}$. For $i\in\mathcal{A}$, boundedness of $\{\mu^{(\nu)}\}$ follows from LICQ: introduce
\[
\mathsf{J}(x,B):=\bigl[\, \{\nabla_x s_i(x,B)\}_{i\in\mathcal{A}}\ \ \ \ \{\nabla_B s_i(x,B)\}_{i\in\mathcal{A}} \,\bigr].
\]
LICQ implies $\sigma_{\min}(\mathsf{J}(x^\star,B^\star))>0$, hence in a neighborhood $\|\sum_{i\in\mathcal{A}} \mu_i^{(\nu)}[\nabla_x s_i;\nabla_B s_i]\|\ge \underline{\sigma}\,\|\mu_{\mathcal{A}}^{(\nu)}\|$. Taking norms and using boundedness of $\nabla V,\nabla s_i$ yields $\|\mu_{\mathcal{A}}^{(\nu)}\|\le C$ for some $C$ independent of $\nu$. Extract a convergent subsequence $\mu^{(\nu)}\to\mu\ge 0$; pass to the limit to obtain KKT stationarity and complementarity (inactive $\Rightarrow$ $\mu_i=0$, active $\Rightarrow$ $s_i\to 0$). 
\end{proof}

\section{Periodic Rigidity}\label{sec:rigidity}
\paragraph{Notation and trivial motions (disambiguation).}
On the active set $\mathcal{A}$, write $r_{ij,t}=x_i-x_j-t$ and $n_{ij,t}=r_{ij,t}/\|r_{ij,t}\|$.
We \emph{define}
\[
A \ :=\ \dot B\,B^{-1}\quad\text{(cell velocity/strain rate; dimensionless)}.
\]
A periodic infinitesimal motion is a pair $(u,A)$ with $u=(u_1,\dots,u_N)\in(\R^n)^N$ and $A\in\R^{n\times n}$ satisfying
\begin{equation}\label{eq:inf}
r_{ij,t}^\top\bigl(u_i-u_j- A\,r_{ij,t}\bigr)=0\qquad \forall (i,j,t)\in\mathcal{A}.
\end{equation}
Trivial (Euclidean) motions form the subspace
\[
\Triv \ :=\ \bigl\{\, (u,A):\ u_i=c+\Omega x_i,\ A=\Omega,\ \Omega^\top=-\Omega,\ c\in\R^n \,\bigr\}.
\]

\begin{definition}[Prestress and stress energy]
An \emph{equilibrium stress} is a set $\omega=\{\omega_{ij,t}\}_{\mathcal{A}}$ with $\omega_{ij,t}=\omega_{ji,-t}$. Its quadratic form on motions is
\begin{equation}\label{eq:Q}
Q_\omega(u,A)\ :=\ \sum_{(i,j,t)\in\mathcal{A}} \omega_{ij,t}\,\bigl(n_{ij,t}^\top(u_i-u_j-A\,r_{ij,t})\bigr)^2.
\end{equation}
We say the framework is \emph{strictly prestress stable} if there exists an equilibrium stress $\omega$ with $Q_\omega$ positive definite on the subspace $\{(u,A):\eqref{eq:inf}\text{ holds}\}/\Triv$.
\end{definition}

\begin{theorem}[Prestress stability $\Rightarrow$ periodic rigidity]\label{thm:periodic-prestress}
If the framework is strictly prestress stable, then it is periodically infinitesimally rigid; i.e., any $(u,A)$ satisfying \eqref{eq:inf} is trivial modulo $\Triv$.
\end{theorem}

\begin{proof}
By strict prestress stability, $Q_\omega$ is positive definite on the nontrivial motion subspace; hence the only motion satisfying \eqref{eq:inf} with $Q_\omega=0$ is trivial (modulo $\Triv$). See \cite{Connelly82,ConnellyWhiteley96,BorceaStreinu10,MalesteinTheran13}.
\end{proof}

\begin{remark}[KKT multipliers versus strict prestress stability]
At a KKT point with strictly positive contact multipliers $\mu_{ij,t}>0$ on $\mathcal{A}$, $\omega=\mu$ is an equilibrium stress. However, positivity of multipliers alone does not imply $Q_\omega\succ 0$ on nontrivial motions; certification requires a constrained spectral check or SOSC-type curvature that rules out nontrivial flats.
\end{remark}

\section{Discrete-Time Dynamics and the \SPIT{} method}\label{sec:dynamics}
We evolve $x$ (and optionally $B$) by a damped velocity-Verlet scheme on the interior barrier $U_\nu$.
We refer to the resulting algorithm as the \emph{Spectral-Projected Interior Trajectory} method (\SPIT): a barrier-driven dynamics with feasibility projection and optional spectral nudges.

Let $\eta>0$ be damping and $\Delta t>0$ a time step.
The corrected update is:
\begin{equation}\label{eq:upd12}
\boxed{
\begin{aligned}
v^{k+\frac12} &= v^k - \frac{\eta\Delta t}{2}\, v^k - \frac{\Delta t}{2}\,\nabla U_\nu(x^k,B^k),\\
x^{k+\frac12} &= x^k + \Delta t\, v^{k+\frac12},\\
v^{k+1} &= \left(1 - \frac{\eta\Delta t}{2}\right) v^{k+\frac12} - \frac{\Delta t}{2}\,\nabla U_\nu(x^{k+\frac12},B^k),\\
x^{k+1} &= x^{k+\frac12} \quad\text{(before projection)}.
\end{aligned}
}
\tag{12}
\end{equation}

\begin{assumption}[Strict-feasibility safeguard]\label{ass:safeguard}
There exists $\delta\in(0,1)$ and a finite neighbor set $\ShiftSet$ such that all iterates satisfy $s_{ij,t}(x^k,B^k)\ge \delta$ for all $(i,j,t)$ used in $U_\nu$; if needed, we project minimally along contact normals to enforce $s\ge \delta$ after each step.
\end{assumption}

\begin{assumption}[Local smoothness]\label{ass:smooth}
On the safeguarded set, $\|\nabla U_\nu(x,B)-\nabla U_\nu(y,B)\|\le \Llip\|x-y\|$.
\end{assumption}

\begin{proposition}[Discrete (unprojected) Lyapunov descent]\label{prop:discrete-descent}
Under Assumptions~\ref{ass:safeguard}--\ref{ass:smooth}, if
\begin{equation}\label{eq:explicit-bounds}
0<\eta\Delta t<2
\qquad\text{and}\qquad
\Llip\,\Delta t^2\ \le\ \tfrac12,
\end{equation}
then with
\begin{equation}\label{eq:Ed}
\mathcal{E}^k := U_\nu(x^k,B^k) + \tfrac12 \|v^k\|^2 + \tfrac{\gamma}{2}\|x^k-x^{k-1}\|^2,\qquad 
\gamma=\tfrac{1}{\Delta t^2}-\tfrac{\Llip}{2}\ \ (\ge 0 \text{ by } \eqref{eq:explicit-bounds}),
\end{equation}
we have $\mathcal{E}(x^{k+1},v^{k+1})\le \mathcal{E}(x^k,v^k)$ \emph{before} feasibility projection.
\end{proposition}

\subsection{Feasibility projection: energy bounds and an \texorpdfstring{$\mathcal{E}$}{E}-nonexpansive variant}\label{sec:proj-energy}
Let $(x,v,B)$ denote the tentative state after \eqref{eq:upd12} and before projection; let $(x^+,B^+)$ be the position/cell after projection (Step~5), and $v^+=v$ (kept unchanged).

\paragraph{A general bound.}
Define the composite gradient
\begin{equation}\label{eq:gbar}
\bar g_x(x,B):=\nabla_x U_\nu(x,B)+\gamma\,(x-x^{k-1}) .
\end{equation}
For any spatial displacement $d_x$ (with $B$ fixed),
\begin{equation}\label{eq:Ebound-generic}
\mathcal{E}(x+d_x,v,B)-\mathcal{E}(x,v,B)
\ \le\ \inner{\bar g_x(x,B)}{d_x}
+\frac{\Llip+\gamma}{2}\,\|d_x\|^2 .
\end{equation}

\paragraph{\(\mathcal{E}\)-nonexpansive projection (x-only).}
Let $C_\delta:=\{y:\ s_{ij,t}(y,B)\ge\delta\ \forall(i,j,t)\}$ and define
\[
M_{x,B}(y):=U_\nu(x,B)+\inner{\nabla_x U_\nu(x,B)}{y-x}+\tfrac{\Llip}{2}\|y-x\|^2+\tfrac{\gamma}{2}\|y-x^{k-1}\|^2 .
\]
Project onto the linearized feasible set
\begin{equation}\label{eq:E-proj}
x^{+}\ \in\ \argmin_{y\in \tilde C_\delta(x,B)} M_{x,B}(y),\quad
\tilde C_\delta(x,B):=\bigl\{y:\ s_{ij,t}(x,B)+\inner{\nabla_x s_{ij,t}(x,B)}{y-x}\ \ge\ \delta\ \ \forall(i,j,t)\bigr\}.
\end{equation}
Then
\begin{equation}\label{eq:nonexpansive}
\mathcal{E}(x^{+},v,B)\ \le\ M_{x,B}(x^{+})+\tfrac12\|v\|^2\ \le\ M_{x,B}(x)+\tfrac12\|v\|^2\ =\ \mathcal{E}(x,v,B).
\end{equation}

\paragraph{\(\mathcal{E}\)-nonexpansive \emph{joint} projection (x and B).}\label{sec:bupdate}
When updating $B$ jointly, use the majorizer
\begin{align}
M_{x,B}(y,H)\ :=\ &U_\nu(x,B)+\inner{\nabla_x U_\nu(x,B)}{y-x}+\langle \nabla_B U_\nu(x,B), H\rangle \nonumber\\
&\quad +\tfrac{\Llipx}{2}\|y-x\|^2+\tfrac{\LlipB}{2}\|H\|_F^2+\tfrac{\gamma}{2}\|y-x^{k-1}\|^2. \label{eq:joint-majorizer}
\end{align}
Linearize each constraint using (with $t=Bz$)
\[
\nabla_x s_{ij,t}(x,B)=2r_{ij,t}\,[\dots,+I,-I,\dots],\qquad
\nabla_B s_{ij,t}(x,B)=-2\,r_{ij,t}\,z^\top.
\]
Solve the QP
\begin{equation}\label{eq:E-proj-joint}
\min_{(y,H)}\ M_{x,B}(y,H)\ \ \text{s.t.}\ \ s_{ij,t}(x,B)+\inner{\nabla_x s_{ij,t}(x,B)}{y-x}+\langle \nabla_B s_{ij,t}(x,B),H\rangle \ \ge\ \delta\ \ \forall(i,j,t),
\end{equation}
and set $(x^+,B^+)=(y^\star,B+H^\star)$.
\begin{lemma}[Joint nonexpansiveness]\label{lem:joint-nonexpansive}
The update $(x^+,B^+)$ from \eqref{eq:E-proj-joint} satisfies $\mathcal{E}(x^+,v,B^+)\le \mathcal{E}(x,v,B)$.
\end{lemma}
\begin{proof}
Same majorization argument as \eqref{eq:nonexpansive}, noting $\mathcal{E}$ contains no kinetic term in $B$ (we keep $v$ unchanged).
\end{proof}

\begin{proposition}[Projected Lyapunov descent (full step)]\label{prop:projected-descent}
Under the assumptions of Proposition~\ref{prop:discrete-descent}, if the feasibility projection is performed via \eqref{eq:E-proj} (x-only) or \eqref{eq:E-proj-joint} (x\&B), then the \emph{post-projection} energy satisfies
\[
\mathcal{E}(x^{k+1},v^{k+1},B^{k+1})\ \le\ \mathcal{E}(x^{k+\frac12},v^{k+1},B^{k})\ \le\ \mathcal{E}(x^{k},v^{k},B^{k}).
\]
\end{proposition}

\begin{theorem}[Local linear convergence]\label{thm:local-linear}
Let $x^\star$ be a strict local minimizer of $U_\nu(\cdot,B^\star)$ with projected Hessian $\nabla^2_{xx} U_\nu(x^\star,B^\star)\succeq m I$ ($m>0$) on the feasible manifold. Under Assumptions~\ref{ass:safeguard}--\ref{ass:smooth}, if $0<\eta\Delta t<2$ and
\begin{equation}\label{eq:ll-explicit}
\Delta t\ <\ \frac{2}{\sqrt{\Llip+m}},
\end{equation}
then $(x^k,v^k)\to(x^\star,0)$ linearly for $x^0$ near $x^\star$. With joint nonexpansive projection, the same rate holds.
\end{theorem}

\subsection{Estimating \texorpdfstring{$\Llip$}{L\_lip} and \texorpdfstring{$m$}{m} by HVPs}\label{sec:hvp}
Write each lattice shift as $t=Bz$ with $z\in\Z^n$ from a finite index set $\mathcal{Z}$ corresponding to $\ShiftSet$.
For $c=(i,j,z)$, set $r_c=x_i-x_j-Bz$, $s_c=\|r_c\|^2-4$, and denote $\phi=\phi_\nu$, $\phi'=\phi_\nu'(s_c)$, $\phi''=\phi_\nu''(s_c)$.

\paragraph{HVP in $x$ (holding $B$ fixed).}
For a direction $p=(p_1,\dots,p_N)$ with $p_i\in\R^n$,
\begin{align}\label{eq:hvp-x}
\big[\nabla^2_{xx}U_\nu\,p\big]_i
&=\sum_{j,z}\Big( 4\phi''\,\inner{r_c}{p_i-p_j}\,r_c\ +\ 2\phi'\,(p_i-p_j)\Big),
\end{align}
with the opposite sign on the $j$-block.

\paragraph{Joint HVP in $(x,B)$.}
For matrix directions $H\in\R^{n\times n}$ and $p$ as above, the mixed and $B$-blocks are
\begin{align}\label{eq:hvp-xb}
\big[\nabla^2_{xB}U_\nu\,[p,H]\big]_i
&=\sum_{j,z}\Big( -4\phi''\,\inner{r_c}{p_i-p_j}\,(Hz)\ -\ 2\phi'\,(Hz)\Big),\\
\label{eq:hvp-b}
\nabla^2_{BB}U_\nu[H]
&=\sum_{i<j,z}\Big( 4\phi''\,\inner{r_c}{Hz}\,r_c\,z^\top\ +\ 2\phi'\,(Hz)\,z^\top\Big).
\end{align}
Power iterations with \eqref{eq:hvp-x} estimate $\hat \Llip\approx\lambda_{\max}$ (for $x$-updates). If $B$ is also updated, use \eqref{eq:hvp-x}--\eqref{eq:hvp-b} to estimate $(\hat \Llipx,\hat \LlipB)$.

\subsection{Algorithmic specification (mode lifting, nudges, step sizes)}\label{sec:algo}
\paragraph{Mode lifting.}
Given the contact graph $G^k=(V,E)$ at $x^k$ with unit normals $n_{ij}$, lift a graph mode $v\in\R^N$ (e.g.\ the $\lambda_2$-eigenvector) to geometry by
\begin{equation}\label{eq:lift}
\delta x_i \;=\; \frac{1}{\deg(i)} \sum_{j:(i,j)\in E} (v_i - v_j)\, n_{ij}.
\end{equation}

\paragraph{Admissible nudge and \(\mathcal{E}\)-safe step.}
Let $\mathcal{N}$ include contacts and $\epsilon$-near neighbors $2<\|x_i-x_j\|\le 2+\epsilon$. Define
\begin{equation}\label{eq:alpha}
\alpha_{\max} = \min_{(i,j)\in \mathcal{N}} \frac{\|x_i-x_j\|-2}{|( \delta x_i - \delta x_j )^\top n_{ij}| + 10^{-12}}.
\end{equation}
With $\bar g_x$ from \eqref{eq:gbar}, flip $\delta x$ so that $\inner{\bar g_x}{\delta x}\le 0$ and take the \(\mathcal{E}\)-safe step
\begin{equation}\label{eq:alphaEsafe}
\alpha_{\mathcal{E}}\ :=\ \min\!\left\{\alpha_{\max},\ \frac{2\,\bigl(-\inner{\bar g_x}{\delta x}\bigr)}{(\Llip+\gamma)\,\|\delta x\|^2}\right\}.
\end{equation}
Enforce $s\ge \delta$ afterward using \eqref{eq:E-proj} or \eqref{eq:E-proj-joint}.

\paragraph{Step sizes and damping (explicit, non-circular).}
Choose
\[
\eta\Delta t\in(0.5,1.5),\qquad
\Delta t\ \le\ \min\Bigl\{ \frac{1}{\sqrt{2\Llip}},\ \frac{c}{\sqrt{\Llip+m}} \Bigr\}\quad\text{with }c<2.
\]
Estimate $\Llip,m$ via HVPs (Sec.~\ref{sec:hvp}); if $\mathcal{E}^{k+1}>\mathcal{E}^{k}$, set $\Delta t\leftarrow \Delta t/2$ and redo the step.

\paragraph{\(\mathcal{E}\)-safe spectral-nudge thresholds and cadence.}
Trigger nudges only if $\lambda_2(G^k)<\tau$, where
\[
\tau\ :=\ \kappa\,\operatorname{median}\bigl\{\lambda_2(G^{k'})\ :\ k'\in\{k-W,\dots,k\}\bigr\}\cdot \min\!\left\{1,\ \frac{\hat m}{\hat \Llip}\right\},
\]
with window $W\in[10,50]$ and $\kappa\in[0.2,0.4]$. Use cadence $K\ge \lceil 4/(\eta\Delta t)\rceil$ so damping dissipates transients between nudges.

\paragraph{Feasibility projection (with \(\mathcal{E}\)-guard).}
After the tentative update, attempt one-pass GS. If any constraint still violates $s\ge\delta$ or if $\mathcal{E}$ failed to drop, solve \eqref{eq:E-proj} (or \eqref{eq:E-proj-joint}) and accept its output.

\section{Minimal schematic illustrations and reproducibility}\label{sec:minimal-illustrations}
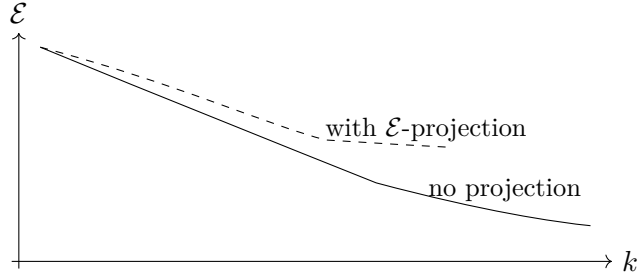
\begin{figure}[htbp]
\centering
\begin{tikzpicture}[scale=0.95]
\draw[->] (-0.1,0) -- (8.3,0) node[right] {$k$};
\draw[->] (0,-0.1) -- (0,3.2) node[above] {$\mathcal{E}$};
\draw (0.3,3) .. controls (2,2.3) and (3.5,1.7) .. (5,1.1) .. controls (6.5,0.7) and (7.5,0.55) .. (8,0.5);
\draw[dashed] (0.3,3) .. controls (2.2,2.5) and (3.5,1.9) .. (4.3,1.7) .. controls (5.1,1.65) and (5.6,1.62) .. (6,1.6);
\node at (6.8,1.0) {\small no projection};
\node at (5.7,1.85) {\small with $\mathcal{E}$-projection};
\end{tikzpicture}
\caption{Schematic (not data): monotone $\mathcal{E}$ descent with (dashed) and without (solid) explicit \(\mathcal{E}\)-nonexpansive projection.}
\label{fig:schematic-energy}
\end{figure}

\paragraph{Reproducibility stub.} Minimal testbed: $n=2$, $N\in\{32,64\}$, random jitter around hexagonal lattice; $\delta=10^{-3}$, $\nu=10^{-2}$, $\eta\Delta t\in[0.8,1.2]$, $\Delta t$ from \eqref{eq:explicit-bounds}; $\ShiftSet$ the shifts $Bz$ with $\|z\|_\infty\le 1$. Use x-only updates first; enable joint projection \eqref{eq:E-proj-joint} every 10 steps; trigger nudges with $\kappa=0.3$, $W=20$, $K=10$. Log: $\mathcal{E}^k$, minimum slack, $\lambda_2(G^k)$.

\section{Limitations and Counterpoints}
We emphasize mathematical guarantees and do not report measured experiments in this manuscript. Spectral indicators stabilize dynamics but do not certify geometric optimality; barrier modeling is analytical and must be paired with feasibility safeguards; discrete guarantees are local and require smoothness on a strict-feasibility neighborhood.

\appendix
\section{Deferred proofs}\label{app:proofs}

\subsection{Proof of Lemma~\ref{lem:Lipschitz}}\label{app:proof-lip}
Let $r_{ij,t}=x_i-x_j-t$ and $s=\|r\|^2-4$. For $i<j$ and fixed $t\in\ShiftSet$, the contribution to $\nabla_x U_\nu$ is
$g_{ij,t}(x)=\phi'(s)\,(2r)\,[\dots,+I \text{ at } i,\dots,-I \text{ at } j,\dots]$.
Hence the block-Hessian (w.r.t.\ $x$) is
$\nabla^2_{x} (\phi\circ s) = 4\phi''(s)\,rr^\top + 2\phi'(s)\,\Id_{ij}$,
where $\Id_{ij}$ is the standard Laplacian block on indices $i,j$. On $s\in[\delta,S]$,
$|\phi'(s)|\le \nu(\tfrac{1}{\delta}+\tfrac{S}{\delta})$ and $\phi''(s)\le \nu(\tfrac{1}{\delta^2}+\tfrac{1}{\delta})$.
With $\|r\|\le R$ and finite summation (A2) we obtain $\|\nabla^2 U_\nu\|\le \Llip$, hence the Lipschitz bound.

\subsection{Proof of Theorem~\ref{thm:barrier-consistency}}\label{app:proof-barrier-kkt}
By (A5) each $\Phi_\nu$ attains a minimizer on a compact feasible set modulo gauge. FOC at $(x_\nu,B_\nu)$ read
$\nabla_B V(B_\nu) - \sum \frac{\nu}{s_{ij,t}(x_\nu,B_\nu)} \nabla_B s_{ij,t}(x_\nu,B_\nu)=0$
and
$\sum \frac{\nu}{s_{ij,t}(x_\nu,B_\nu)} \nabla_x s_{ij,t}(x_\nu,B_\nu)=0$.
Set $\mu_{ij,t}^{(\nu)}=\nu/s_{ij,t}(x_\nu,B_\nu)\ge 0$. Under LICQ (A6) the multipliers are bounded on the active set and vanish on inactive constraints at the limit. Extract a convergent subsequence $\mu_{ij,t}^{(\nu)}\to\mu_{ij,t}\ge 0$; pass to the limit to obtain KKT with complementarity. See \cite{ForsgrenGillWright02}.

\subsection{Proof of Proposition~\ref{prop:discrete-descent}}
Let $g^k=\nabla_x U_\nu(x^k,B^k)$ and $a^k:=v^{k+\frac12}=v^k-\tfrac{\eta\Delta t}{2}v^k-\tfrac{\Delta t}{2}g^k$. $\Llip$-smoothness in $x$ gives
$U_\nu(x^{k+\frac12},B^k)\le U_\nu(x^k,B^k)+\Delta t\,\inner{g^k}{a^k}+\tfrac{\Llip\Delta t^2}{2}\|a^k\|^2$.
Compute the kinetic variation using $v^{k+1}=(1-\tfrac{\eta\Delta t}{2})a^k - \tfrac{\Delta t}{2}g^{k+\frac12}$ and $\|g^{k+\frac12}-g^k\|\le \Llip\Delta t\|a^k\|$.
Collect terms and use $\gamma=\tfrac{1}{\Delta t^2}-\tfrac{\Llip}{2}$ to obtain
$\mathcal{E}(x^{k+\frac12},v^{k+1},B^k)-\mathcal{E}(x^k,v^k,B^k) \le -\tfrac{\eta\Delta t}{2}\|a^k\|^2 + (\tfrac{\Llip\Delta t^2}{2}-\tfrac12)\|a^k\|^2\le 0$
under \eqref{eq:explicit-bounds}.

\subsection{Proof of Theorem~\ref{thm:local-linear}}
Let $H=\nabla^2_{xx} U_\nu(x^\star,B^\star)\succeq mI$. In an eigen-direction with curvature $\lambda\in[m,\Llip]$ the linearized recursion reduces to
$e^{k+1}-\alpha e^k+\beta e^{k-1}=0$ with $\alpha=2-\tfrac{\eta\Delta t}{2}-\tfrac{\lambda\Delta t^2}{2}$ and $\beta=1-\tfrac{\eta\Delta t}{2}$.
The Jury conditions for $|z|<1$ hold uniformly for $\lambda\in[m,\Llip]$ when $0<\eta\Delta t<2$ and $\Delta t<2/\sqrt{\Llip+m}$, yielding linear convergence. Joint projection does not increase $\mathcal{E}$ and thus preserves the rate.

\section*{Acknowledgements}
We thank the community for discussions on barrier methods, rigidity, and discrete dynamics.

\end{document}